\documentclass[11pt,reqno]{article}
\usepackage{latexsym, amsmath, amssymb, amsthm, a4, epsfig}
\usepackage{graphicx}
\usepackage{mathrsfs}
\usepackage{fancyhdr}

\newtheorem{theorem}{Theorem}[section]

\newtheorem{lemma}[theorem]{Lemma}
\newtheorem{prop}[theorem]{Proposition}

\newcommand{\nm}{\noalign{\smallskip}}
\newcommand{\ds}{\displaystyle}
\newcommand{\pf}{\noindent {\sl Proof}. \ }
\newcommand{\p}{\partial}

\newcommand{\eqnref}[1]{(\ref {#1})}

\newcommand{\Rbb}{\mathbb{R}}

\newcommand{\Zbb}{\mathbb{Z}}

\newcommand{\la}{\langle}
\newcommand{\ra}{\rangle}

\newcommand{\Acal}{\mathcal{A}}

\newcommand{\Kcal}{\mathcal{K}}

\newcommand{\Scal}{\mathcal{S}}

\def\Be{{\bf e}}

\def\Bx{{\bf x}}
\def\By{{\bf y}}

\newcommand{\Ga}{\alpha}

\newcommand{\Gd}{\delta}

\newcommand{\Gvf}{\varphi}

\newcommand{\Gl}{\lambda}
\newcommand{\Gn}{\eta}
\newcommand{\Gm}{\mu}
\newcommand{\Gv}{\nu}
\newcommand{\Gp}{\pi}
\newcommand{\Gt}{\theta}

\newcommand{\Gs}{\sigma}

\newcommand{\Go}{\omega}
\newcommand{\Gx}{\xi}
\newcommand{\Gy}{\psi}

\newcommand{\GD}{\Delta}

\newcommand{\GG}{\Gamma}

\newcommand{\GO}{\Omega}

\newcommand{\beq}{\begin{equation}}
\newcommand{\eeq}{\end{equation}}

\def\ol{\overline}

\numberwithin{equation}{section}
\numberwithin{figure}{section}

\begin{document}

\title{Spectral structure of the Neumann--Poincar\'e operator on tori\thanks{\footnotesize This work was
supported by NRF grants No. 2016R1A2B4011304 and 2017R1A4A1014735, and by A3 Foresight Program among China (NSF), Japan (JSPS), and Korea (NRF 2014K2A2A6000567).}}


\author{Kazunori Ando\thanks{Department of Electrical and Electronic Engineering and Computer Science, Ehime University,
Ehime 790-8577, Japan (ando@cs.ehime-u.ac.jp).} \and Yong-Gwan Ji\thanks{Department of Mathematics and Institute of Applied Mathematics, Inha University, Incheon 22212, S. Korea (22151063@inha.edu, hbkang@inha.ac.kr, k.goe.dai@gmail.com).} \and Hyeonbae Kang\footnotemark[3] \and Daisuke Kawagoe\footnotemark[3]  \and Yoshihisa Miyanishi\thanks{Center for Mathematical Modeling and Data Science, Osaka University, Osaka 560-8531, Japan (miyanishi@sigmath.es.osaka-u.ac.jp).}}

\maketitle

\begin{abstract}
We address the question whether there is a three-dimensional bounded domain such that the Neumann--Poincar\'e operator defined on its boundary has infinitely many negative eigenvalues. It is proved in this paper that tori have such a property. It is done by decomposing the Neumann--Poincar\'e operator on tori into infinitely many self-adjoint compact operators on a Hilbert space defined on the circle using the toroidal coordinate system and the Fourier basis, and then by proving that the numerical range of infinitely many operators in the decomposition has both positive and negative values.
\end{abstract}



\section{Introduction}

The goal of this paper is to prove the following theorem.
\begin{theorem}\label{main1}
The Neumann-Poincar\'e operator on tori has infinitely many negative eigenvalues as well as infinitely many positive ones.
\end{theorem}

To demonstrate novelty of this result we briefly review a history of the spectral theory of the Neumann-Poincar\'e (abbreviated by NP) operator.

The NP operator is an integral operator naturally arising when solving classical boundary value problems using layer potentials. It is defined on the boundary $\p\GO$ of a bounded domain $\GO$ in $\Rbb^d$ ($d=2,3$). Precisely, it is defined by
\beq\label{introkd}
\Kcal_{\p \GO}^* [\Gvf] (\Bx) =
\mbox{p.v.} \frac{1}{\Go_d} \int_{\p \GO} \frac{(\Bx - \By)\cdot\nu_\Bx }{|\Bx - \bold{y}|^d} \Gvf(\bold{y})\,d\Gs(\bold{y}), \quad \Bx \in \p \GO,
\eeq
where $\nu_\Bx$ denotes the outward unit normal vector to $\p \GO$ at $\Bx$ and $\Go_d=2\pi$ if $d=2$, $\Go_d=4\pi$ if $d=3$. Either $\Kcal_{\p \GO}^*$ or its adjoint $\Kcal_{\p \GO}$ (in $L^2(\p\GO)$) is called the NP operator on $\p\GO$. $\Kcal_{\p \GO}$ is frequently called the double layer potential.

Observe that the integral kernel of $\Kcal_{\p \GO}^*$ is the normal derivative of the fundamental solution to the Laplacian
\beq\label{gammacond}
-\GG (\Bx) =
\begin{cases}
\ds \frac{1}{2\pi} \ln |\Bx|\;, \quad & d=2 \;, \\ \nm \ds
-\frac{1}{4\pi} |\Bx|^{-1}\;, \quad & d = 3 \;.
\end{cases}
\eeq
The single layer potential $\Scal_{\p \GO} [\Gvf]$ of a density function $\Gvf \in L^2(\p \GO)$ is defined by
\beq\label{singledef}
\Scal_{\p \GO} [\Gvf] (\Bx) := \int_{\p \GO} \GG (\Bx - \bold{y}) \Gvf (\bold{y}) \, d\Gs(\bold{y}), \quad \Bx \in \Rbb^d .
\eeq
It is also common to use $-\GG(\Bx)$ for the integral kernel to define the single layer potential. We take the definition \eqnref{singledef} in this paper so that the single layer potential becomes a positive operator. The connection between the NP operator and the single layer potential is given by the jump relation (see, for example, \cite{AmKa07Book2, Fo95}):
\beq\label{singlejump}
\p_\nu \Scal_{\p \GO} [\Gvf] \Big|_\pm (\Bx) = \biggl( \mp \frac{1}{2} I - \Kcal_{\p \GO}^* \biggr) [\Gvf] (\Bx),
\quad \Bx \in \p \GO \;,
\eeq
where $\p_{\nu}$ denotes the outward normal derivative and the subscripts $\pm$ indicate the limit from outside and inside $\GO$, respectively.

The relation \eqnref{singlejump} shows that, for example, to solve the Neumann problem, $\GD u=0$ in $\GO$ and $\p_\nu u = f$ on $\p\GO$, it suffices to have $u:= \Scal_{\p \GO} [\Gvf]$ in $\GO$, where $\Gvf$ is the solution of the following integral equation:
\beq\label{fredholm}
\biggl( \frac{1}{2} I - \Kcal_{\p \GO}^* \biggr) [\Gvf]=f \quad\mbox{on } \p\GO.
\eeq
This kind of approach for solving boundary value problems traces back to C. Neumann \cite{Neumann-87} and Poincar\'e \cite{Poincare-AM-87} as the name of the operator suggests. If $\p\GO$ is smooth ($C^{1,\Ga}$ for some $\Ga>0$ to be precise), then $\Kcal_{\p \GO}^*$ is compact on $L^2(\p\GO)$ (and on $H^{-1/2}(\p\GO)$, the Sobolev space of order $-1/2$), and hence the Fredholm index theory can be applied to solve \eqnref{fredholm}. On the other hand, if $\p\GO$ is merely Lipschitz, say if it has a corner, then $\Kcal_{\p \GO}^*$ is a singular integral operator which has been one of central subjects of mathematical research in the last century. For example, the $L^2$-boundedness was proved in the seminal paper \cite{CMM82} and solvability of \eqnref{fredholm} was established in \cite{Verch-JFA-84}.

Note that $\Kcal_{\p \GO}^*$ is not self-adjoint on $L^2(\p\GO)$, namely, $\Kcal_{\p \GO}^* \neq \Kcal_{\p \GO}$, unless $\GO$ is a disk or a ball \cite{Lim}. However, in \cite{KPS} where Poincar\'e's work was revived in modern language, it is revealed that $\Kcal_{\p \GO}^*$ can be realized as a self-adjoint operator on $H^{-1/2}(\p\GO)$ by introducing an inner product with the single layer potential. Let, for $\Gvf, \psi \in H^{-1/2}(\p\GO)$,
\beq\label{innerproduct}
\la \Gvf, \psi \ra_*:= (\Gvf, \Scal_{\p\GO}[\psi]),
\eeq
where $( \cdot, \cdot)$ is the $H^{-1/2}-H^{1/2}$ pairing. Since $\Scal_{\p\GO}$ maps $H^{-1/2}(\p\GO)$ into $H^{1/2}(\p\GO)$, $\la \cdot, \cdot \ra_*$ is well-defined. In fact, it is an inner product on $H^{-1/2}(\p\GO)$ in three dimensions, and the norm induced by $\la \cdot, \cdot \ra_*$ is actually equivalent to $H^{-1/2}$-norm (see, for example, \cite{KKLSY}). In two dimensions, $\Scal_{\p\GO}$ may have one-dimensional kernel \cite{Verch-JFA-84}, but it can be remedied so that $\la \cdot, \cdot \ra_*$ is an inner product (see, e.g., \cite{AK}). Then one can use the Plemelj's symmetrization principle, which states
\beq
\Scal_{\p\GO} \Kcal_{\p \GO}^* = \Kcal_{\p \GO}\Scal_{\p\GO},
\eeq
to symmetrize $\Kcal_{\p \GO}^*$, that is,
\beq\label{symmetry}
\la \Kcal_{\p \GO}^*[\Gvf], \psi \ra_* = \la \Gvf, \Kcal_{\p \GO}^*[\psi] \ra_*.
\eeq
Now, if $\p\GO$ is $C^{1,\Ga}$ for some $\Ga>0$, then $\Kcal_{\p \GO}^*$, as a self-adjoint compact operator, has real eigenvalues converging to $0$. It is worth mentioning that there are some work on convergence rate \cite{AKM, MS} culminated in Weyl's law in three dimensions \cite{Miya}. If $\p\GO$ has a corner, then $\Kcal_{\p \GO}^*$ has continuous spectrum \cite{BZ, HKL, HP, KLY, PP, PP2}.

Lately interest in the spectral properties of the NP operator is growing fast, which is due to their relations to plasmonics: plasmonic resonance occurs at eigenvalues of the NP operator \cite{AMRZ, MFZ} and anomalous localized resonance occurs at the accumulation point of eigenvalues \cite{ACKLM}. However, for all significant progress that has been made, research on NP spectrum (spectrum of the NP operator) is at its early stage and many questions still remain unanswered. The question on negative eigenvalues is one of them.

Unlike two-dimensional NP spectrum, which is symmetric with respect to $0$ except $1/2$ (see, e.g., \cite{HKL, KPS}) and hence has the same number of negative eigenvalues as positive ones, not so many surfaces (boundaries of three-dimensional domains) are known to have negative NP eigenvalues. In fact, NP eigenvalues on spheres are all positive, and Poincar\'e suggested that all the NP eigenvalues are positive even though notion of spectrum did not exist at his time (see \cite{BT, KPS}). It is only in 1994 that the NP operator on an oblate spheroid is shown to have a negative eigenvalue \cite{Ahner}, which was the first example of surfaces with a negative NP eigenvalue. We emphasize the oblate spheroid considered in the above mentioned paper is thin, and negativity of an eigenvalue is shown numerically. Furthermore, we do not know how many negative eigenvalues  there are. We also mention that NP eigenvalues on ellipsoids can be found explicitly using Lam\'e functions for which we also refer to \cite{AA, FK, Mart, Ritt}. However, it seems quite difficult to see whether there are negative eigenvalues and how many they are if they exist. Recently, a concavity condition is found, which is sufficient for the NP operator on either the boundary of the domain or its inversion to have a negative eigenvalue \cite{JK}. For example, this condition is fulfilled if there is a point on the boundary where the Gaussian curvature is negative. Thus a natural follow-up question is whether there is a surface admitting infinitely many negative NP eigenvalues.

The study on negative NP eigenvalues in this paper is motivated by a historic reason as mentioned above. In addition, negative NP eigenvalues have some implication on numerical schemes. For example, the optimal parameter for an iterative scheme to solve an exterior Neumann problem was found under the condition that NP eigenvalues are all non-negative \cite[pp. 152--153]{CK}. However, deep understanding on implications of negative NP eigenvalues seems still missing.

Lacking general theory for negative eigenvalues, we seek examples of surfaces with infinitely many negative NP eigenvalues, and tori are candidates. The reason to choose tori as candidates is twofold. One is that a significant portion of tori has a negative Gaussian curvature, and the other is that they have good symmetries to be exploited. In fact, it is conjectured in \cite{Miya} that the NP operator on tori has infinitely many negative eigenvalues. Theorem \ref{main1} resolves it. We emphasize that this is the first example of surfaces on which the NP operator has infinitely many negative eigenvalues.

Theorem \ref{main1} is proved as follows. We decompose the NP operator into infinitely many self-adjoint compact operators on a Hilbert space defined on a circle using the toroidal coordinate system and the Fourier basis. We then show that infinitely many operators in the decomposition have numerical ranges having both positive and negative values, which means that they have both positive and negative eigenvalues. This is proved using the stationary phase method. Since the NP spectrum contains the collection of all eigenvalues of operators in the decomposition, Theorem \ref{main1} follows.

This paper is organized as follows. In section \ref{sec:tor}, we introduce the toroidal coordinate system, and express the single layer potential and the NP operator in terms of that coordinate system. In section \ref{sec:Fou}, we decompose the NP operator by the Fourier expansion with respect to the usual toroidal angle to obtain a series of self-adjoint compact operators, and we show a relation between eigenvalues of these operators and those of the NP operator. In section \ref{sec:main}, we show existence of infinitely many negative eigenvalues of the NP operator as well as infinitely many positive ones.

\section{Toroidal coordinate system and the NP operator} \label{sec:tor}

In this section, we express the single layer potential and the NP operator on a torus in terms of the toroidal coordinate system.

The toroidal coordinate system $(\xi, \eta, \Gvf)$ is given by
\beq
x=\frac{R_0 \sqrt{1-\xi^2}\cos \varphi}{1-\xi \cos \eta}, \quad
y=\frac{R_0 \sqrt{1-\xi^2}\sin \varphi}{1-\xi \cos \eta}, \quad
z=-\frac{R_0 \xi \sin \eta}{1-\xi \cos \eta},
\eeq
where $x$, $y$ and $z$ are ordinary Cartesian coordinates, and $R_0:=\sqrt{r_0^2-a^2}$ is the location of the poloidal axis. The surface $\xi=\mbox{constant}$ is a torus. The parameters $r_0$ and $a$ are the major and minor radii, respectively, of a toroidal system. The variable $\xi$ $(0 < \xi < 1)$ is similar to a minor radius, $\eta$ $(0 \leq \eta <2\pi)$ is a poloidal angle, and $\Gvf$ $(0 \leq \Gvf < 2\pi)$ is the usual toroidal angle (equivalent to the azimuthal angle of standard cylindrical coordinates $(r, \Gvf, z)$, see \cite{B} and the figures therein). The toroidal coordinate system is orthogonal with the scale factors
\beq\label{Jacob}
h_\Gx = \dfrac{R_0}{\sqrt{1 - \Gx^2}(1 - \Gx \cos \Gn)}, \quad h_\Gn = \dfrac{R_0 \Gx}{1 - \Gx \cos \Gn}, \quad h_\Gvf = \dfrac{R_0 \sqrt{1 - \Gx^2}}{1 - \Gx \cos \Gn}.
\eeq

We denote by $\p \GO$ the torus parametrized by $\Gx$, and let $\Bx=(\xi, \eta, \Gvf)$ and $\By=(\xi, \eta', \Gvf')$ be points on $\p\GO$. The above mentioned paper also showed that the fundamental solution is given by
\begin{align}
\dfrac{1}{4\pi} \dfrac{1}{|\Bx - \By|} =& \dfrac{\sqrt{1 - \Gx \cos \Gn} \sqrt{1 - \Gx \cos \Gn^\prime}}{4\pi\sqrt{2} R_0(1 - \Gx^2 \cos(\Gn - \Gn^\prime) - (1 - \Gx^2)  \cos(\Gvf - \Gvf^\prime))^{1/2}} \nonumber \\
=& \dfrac{\Gy(\Gn)^{1/2} \Gy(\Gn')^{1/2}}{4\pi \sqrt{2}R_0 \Gx (\Gm(\Gvf - \Gvf^\prime) - \cos(\Gn - \Gn^\prime))^{1/2}}, \label{fundamental}
\end{align}
where
\beq\label{Gndef}
\Gm(\Gvf - \Gvf^\prime) := \dfrac{1}{\Gx^2} + \left( 1 - \dfrac{1}{\Gx^2} \right) \cos(\Gvf - \Gvf^\prime)
\eeq
and
\beq
\Gy(\Gn) := 1 - \Gx \cos \Gn.
\eeq

We see from \eqnref{Jacob} and \eqnref{fundamental} that the single layer potential $\Scal_{\p \GO}$ defined by \eqnref{singledef} can be expressed as
\beq\label{singledef2}
\Scal_{\p \GO}[f](\Gn, \Gvf) = \int_0^{2\pi} \int_0^{2\pi} s(\Gn, \Gn'; \Gvf - \Gvf') f(\Gn', \Gvf')\,d\Gn' d\Gvf',
\eeq
where
\beq
s(\Gn, \Gn'; \Gvf - \Gvf') := \frac{R_0 \sqrt{1 - \Gx^2} \Gy(\Gn)^{1/2}}{4\pi \sqrt{2} \Gy(\Gn')^{3/2} } \dfrac{1}{ (\Gm(\Gvf - \Gvf^\prime) - \cos(\Gn - \Gn^\prime))^{1/2}}.
\eeq

Similarly, we describe the NP operator in terms of the toroidal coordinate system. In the toroidal coordinate system, the outward unit normal vector $\Gv_\Bx$ takes the form
$$
\Gv_\Bx = \frac{(\cos \Gn - \Gx) \cos \Gvf}{\Gy(\Gn)} \Be_1 + \frac{(\cos \Gn - \Gx) \sin \Gvf}{\Gy(\Gn)} \Be_2 - \frac{\sqrt{1 - \Gx^2} \sin \Gn}{\Gy(\Gn)} \Be_3,
$$
where $\Be_1$, $\Be_2$ and $\Be_3$ are unit vectors directing to $x$-axis, $y$-axis and $z$-axis in Cartesian coordinates, respectively (see \cite{B}). So, we have
\begin{align}
(\Bx - \By) \cdot \Gv_\Bx =& \left( \frac{R_0 \sqrt{1-\xi^2}\cos \varphi}{\Gy(\Gn)} - \frac{R_0 \sqrt{1-\xi^2}\cos \Gvf'}{\Gy(\Gn')} \right) \frac{(\cos \Gn - \Gx) \cos \Gvf}{\Gy(\Gn)} \nonumber\\
&+ \left( \frac{R_0 \sqrt{1-\xi^2}\sin \varphi}{\Gy(\Gn)} - \frac{R_0 \sqrt{1-\xi^2}\sin \Gvf'}{\Gy(\Gn')} \right) \frac{(\cos \Gn - \Gx) \sin \Gvf}{\Gy(\Gn)} \nonumber\\
&+ \left( \frac{R_0 \xi \sin \eta}{\Gy(\Gn)} - \frac{R_0 \xi \sin \eta'}{\Gy(\Gn')} \right) \frac{\sqrt{1 - \Gx^2} \sin \Gn}{\Gy(\Gn)} \nonumber\\
=& \frac{R_0 \sqrt{1 - \Gx^2} \cos \Gn(1 - \cos (\Gvf - \Gvf')) - R_0 \Gx \sqrt{1 - \Gx^2} (\cos (\Gn - \Gn') - \cos (\Gvf - \Gvf'))}{\Gy(\Gn) \Gy(\Gn')} \nonumber\\
=& \frac{R_0 \Gx \sqrt{1 - \Gx^2} (\Gm(\Gvf - \Gvf') - \cos (\Gn - \Gn'))}{\Gy(\Gn) \Gy(\Gn')} - \frac{R_0 \sqrt{1 - \Gx^2} (1 - \cos (\Gvf - \Gvf'))}{\Gx \Gy(\Gn')} \label{num}.
\end{align}
According to \eqnref{Jacob}, \eqnref{fundamental} and \eqnref{num}, the NP operator $\Kcal^*_{\p \GO}$ defined by \eqnref{introkd} takes the form
\beq
\Kcal^*_{\p \GO}[f](\Gn, \Gvf) = \int_0^{2\pi} \int_0^{2\pi} k(\Gn, \Gn'; \Gvf - \Gvf') f(\Gn', \Gvf')\,d\Gn' d\Gvf',
\eeq
where
\begin{align}
k(\Gn, \Gn'; \Gvf - \Gvf')
&= \dfrac{1 - \Gx^2}{8\pi \sqrt{2} \Gx}
\dfrac{\psi(\Gn)^{1/2}}{\psi(\Gn')^{3/2}} \dfrac{1}{(\Gm(\Gvf - \Gvf') - \cos (\Gn - \Gn'))^{1/2}} \nonumber \\
&\quad - \dfrac{1 - \Gx^2}{8\pi \sqrt{2} \Gx^3}
\dfrac{\psi(\Gn)^{3/2}}{\psi(\Gn')^{3/2}} \dfrac{1 - \cos (\Gvf - \Gvf')}{(\Gm(\Gvf - \Gvf') - \cos (\Gn - \Gn'))^{3/2}}.
\end{align}

\section{Decomposition of the NP operator} \label{sec:Fou}

Suppose that $f$ is of the form
\beq\label{fgrel}
f(\Gn, \Gvf) = \psi(\Gn)^{3/2} g(\Gn) e^{i k \Gvf}.
\eeq
In this case, we have
\begin{align*}
\Kcal^*_{\p \GO}[f](\Gn, \Gvf) =& \int_0^{2\pi} \int_0^{2\pi} k(\Gn, \Gn'; \Gvf - \Gvf') \psi(\Gn')^{3/2} g(\Gn') e^{i k \Gvf'}\,d\Gvf' d\Gn'\\
=& \int_0^{2\pi} \left( \int_0^{2\pi} k(\Gn, \Gn'; \Gvf') e^{-i k \Gvf'}\,d\Gvf' \right) \psi(\Gn')^{3/2} g(\Gn') \,d\Gvf' d\Gn' e^{i k \Gvf}.
\end{align*}
Define
\begin{align}
a_k(\Gn, \Gn') :&= \int_0^{2\pi} k(\Gn, \Gn'; \Gvf') e^{-i k \Gvf'}\,d\Gvf' \dfrac{\psi(\Gn')^{3/2}}{\psi(\Gn)^{3/2}} \dfrac{8\pi \sqrt{2} \Gx}{1 - \Gx^2} \nonumber \\
&= \psi(\Gn)^{-1} \int_0^{2\pi} \dfrac{ e^{-i k \Gvf'}}{(\Gm(\Gvf') - \cos(\Gn - \Gn'))^{1/2}} \,d\Gvf' \nonumber \\
& \quad - \dfrac{1}{\Gx^2} \int_0^{2\pi} \dfrac{ (1 - \cos \Gvf') e^{-i k \Gvf'}}{(\Gm(\Gvf') - \cos(\Gn - \Gn'))^{3/2}} \,d\Gvf', \label{akdef}
\end{align}
and define the operator $\Acal_k$ by
$$
\Acal_k[g](\Gn) := \int_0^{2\pi} a_k(\Gn, \Gn') g(\Gn')\,d\Gn'.
$$
Then, we have
\beq \label{relationKA}
\Kcal^*_{\p \GO}[f](\Gn, \Gvf) = \dfrac{1 - \Gx^2}{8\pi \sqrt{2} \Gx} \psi(\Gn)^{3/2} \Acal_k[g](\Gn) e^{i k \Gvf},
\eeq
which implies the following lemma.

\begin{lemma} \label{main2}
If $\Gl$ is an eigenvalue of $\Acal_k$ with an eigenfunction $g$, then $(1 - \Gx^2) \Gl / 8\pi \sqrt{2} \Gx$ is an eigenvalue of $\Kcal^*_{\p \GO}$ with the eigenfunction of the form \eqnref{fgrel}.
\end{lemma}

If $f$ is of the form \eqnref{fgrel}, then we have from \eqnref{singledef2}
\begin{align*}
\Scal_{\p \GO}[f](\Gn, \Gvf) &= \dfrac{R_0 \sqrt{1 - \Gx^2}}{4 \pi \sqrt{2}}\psi(\Gn)^{1/2} \int_0^{2\pi} \int_0^{2\pi} \dfrac{g(\Gn') e^{i k \Gvf'}\,d\Gn' d\Gvf'}{(\Gm(\Gvf - \Gvf') - \cos(\Gn - \Gn'))^{1/2} } \\
&= \dfrac{R_0 \sqrt{1 - \Gx^2}}{4 \pi \sqrt{2}} \psi(\Gn)^{1/2} \int_0^{2\pi} \left( \int_0^{2\pi} \dfrac{ e^{-i k \Gvf'}\,d\Gvf'}{(\Gm(\Gvf') - \cos(\Gn - \Gn'))^{1/2}} \right) g(\Gn')\,d\Gn' e^{i k \Gvf}.
\end{align*}
Define
\beq
s_k(\Gn) := \int_0^{2\pi} \dfrac{ e^{-i k \Gvf'}\,d\Gvf'}{(\Gm(\Gvf') - \cos \Gn)^{1/2}}
\eeq
and
\beq
\Scal_k[g](\Gn) := \int_0^{2\pi} s_k(\Gn - \Gn') g(\Gn') \,d\Gn'.
\eeq
Then, we have
\beq\label{singlerela}
\Scal_{\p \GO}[f](\Gn, \Gvf) = \dfrac{R_0 \sqrt{1 - \Gx^2}}{4 \pi \sqrt{2}} \psi(\Gn)^{1/2} \Scal_k[g](\Gn) e^{i k \Gvf}.
\eeq

Suppose that
\beq\label{fj12}
f_j(\Gn, \Gvf) = \psi(\Gn)^{3/2} g_j(\Gn) e^{i k_j \Gvf}, \quad j = 1, 2.
\eeq
If $f_1$ and $f_2$ are smooth, then we have from \eqnref{innerproduct} and \eqnref{Jacob} that
\begin{align*}
\la f_1, f_2 \ra_* = \int_0^{2\pi} \int_0^{2\pi} f_1(\Gn, \Gvf) \ol{\Scal_{\p \GO}[f_2](\Gn, \Gvf)} \dfrac{R_0^2 \Gx \sqrt{1 - \Gx^2}}{\Gy(\Gn)^2}\,d\Gn d\Gvf.
\end{align*}
It then follows from \eqnref{singlerela} that
\begin{align}
\la f_1, f_2 \ra_* &= \dfrac{{R_0}^3 \Gx(1 - \Gx^2)}{4 \pi \sqrt{2}} \int_0^{2 \pi} e^{i (k_1 - k_2)\Gvf}\,d\Gvf \int_0^{2\pi} g_1(\Gn) \ol{\Scal_{k_2}[g_2](\Gn)} \,d\Gn \nonumber\\
&= \frac{{R_0}^3 \Gx (1 - \Gx^2)}{2 \sqrt{2} } \Gd_{k_1 k_2}\int_0^{2\pi} g_1(\Gn) \ol{\Scal_{k_2}[g_2](\Gn)} \,d\Gn,  \label{innerrela}
\end{align}
where $\Gd_{k_1 k_2}$ is the Kronecker's delta.

Let $T$ be the unit circle and let $H^{s}(T^2)$ be the Sobolev space on the torus $T^2$ equipped with the norm
\beq
\| f \|^2_{s, T^2} := \sum_{k,l =-\infty}^\infty (1+ |k|^2 + |l|^2)^{s} | \hat{f}(k,l)|^2,
\eeq
where $\hat{f}(k,l)$ denotes the double Fourier coefficient of $f$. Then, one can see easily that $H^{s}(\p\GO)$ is equivalent to $H^{s}(T^2)$ for $s=0$ and $s=1$. Then, by interpolation between $s=0$ and $s=1$, we see that $H^{1/2}(\p\GO)$ is equivalent to $H^{1/2}(T^2)$, and hence by duality $H^{-1/2}(\p\GO)$ is equivalent to $H^{-1/2}(T^2)$. Since $\la f, f \ra_*$ is equivalent to $\| f \|_{H^{-1/2}(\p\GO)}$, there is $C>1$ such that
\beq\label{ffequiv}
C^{-1} \| f \|^2_{-1/2, T^2} \le \la f, f \ra_* \le C \| f \|^2_{-1/2, T^2}.
\eeq
It is worthwhile mentioning that $C$ depends on the parameter $\xi$.
We now introduce a Hilbert space on the unit circle: Let $H^{-1/2}(T)$ be the Sobolev space of order $-1/2$ on the unit circle $T$ whose norm is given by
\beq
\| \Gvf \|_{-1/2}^2: = \sum_{l =-\infty}^\infty (1+ |l|^2)^{-1/2} | \hat{\Gvf}(l)|^2,
\eeq
where $\hat{\Gvf}(l)$ denotes the $l$-th Fourier coefficient. We then define $H(T)$ by
\beq
H(T):= \{ g ~|~ \psi^{3/2} g \in H^{-1/2}(T) \}.
\eeq
Then $H(T)$ is a Hilbert space with the norm
\beq
\| g\|_H:= \| \psi^{3/2} g \|_{-1/2}.
\eeq
For $g \in H(T)$, define $f$ by \eqnref{fgrel}. Then there is a constant $C_k$ depending on $k$ such that
$$
C_k^{-1} \| g\|_H \le \| f \|_{-1/2, T^2} \le C_k \| g\|_H.
$$
It then follows from \eqnref{ffequiv} that
\beq\label{equiv1}
C_k^{-1} \| g\|_H \le \| f \|_{H^{-1/2}(\p\GO)} \le C_k \| g\|_H
\eeq
with some different $C_k$.

Since $\Scal_{\p\GO}$ maps $H^{-1/2}(\p\GO)$ into its dual space $H^{1/2}(\p\GO)$ continuously, the relation \eqnref{innerrela} shows that $\Scal_k$ maps $H(T)$ into its dual space $H'(T)$ continuously. Thus we can define
\beq\label{kinner}
\la g_1, g_2 \ra_k := \int_0^{2\pi} g_1(\Gn) \ol{\Scal_k[g_2](\Gn)} \,d\Gn, \quad g_1, g_2 \in H(T),
\eeq
understanding the right-hand side as the $H-H'$ pairing.

\begin{prop}
For each integer $k$, $\la \cdot, \cdot \ra_k$ is an inner product on $H(T)$ and  there is a constant $C_k>1$ depending on $k$ such that
\beq\label{equiv2}
C_k^{-1} \| g\|_H \le \la g, g \ra_k \le C_k \| g\|_H
\eeq
for all $g \in H$. Moreover, $\Acal_k$ is compact and self-adjoint on $H(T)$:
\beq\label{symmetryk}
\la \Acal_k[g_1], g_2 \ra_k = \la g_1, \Acal_k[g_2] \ra_k.
\eeq
\end{prop}

\pf
For $g \in H(T)$, define $f$ by \eqnref{fgrel}. It then follows from \eqnref{innerrela} that
\beq
\la f, f \ra_* = \frac{{R_0}^3 \Gx (1 - \Gx^2)}{2 \sqrt{2} } \la g, g \ra_k.
\eeq
One can easily see from this relation that $\la \cdot, \cdot \ra_k$ is an inner product on $H(T)$. Moreover, since $\la f, f \ra_*$ is equivalent to $\| f \|_{H^{-1/2}(\p\GO)}$, \eqnref{equiv2} follows from \eqnref{equiv1}. Since $\Kcal^*_{\p \GO}$ is compact on $H^{-1/2}(\p\GO)$, \eqnref{relationKA} and \eqnref{equiv1} show that $\Acal_k$ is compact on $H(T)$.

Now we prove that $\Acal_k$ is self-adjoint on $H(T)$. Let $f_1$ and $f_2$ be of the form \eqnref{fj12} with $k_1=k_2=k$. Then \eqnref{symmetry} reads
\beq \label{Plemelj}
\int_{\p \GO} f_1 \ol{\Scal_{\p \GO}[\Kcal^*_{\p \GO}[f_2]]}\,d\Gs = \int_{\p \GO} \Kcal^*_{\p \GO}[f_1] \ol{\Scal_{\p \GO}[f_2]}\,d\Gs.
\eeq
Then, we have from \eqnref{singlerela} 
$$
\Scal_{\p \GO} \Kcal^*_{\p \GO}[f_2](\Gn, \Gvf) = \dfrac{R_0(1 - \Gx^2)^{3/2}}{64 \pi^2 \Gx} \psi(\Gn)^{1/2} \Scal_k[\Acal_k[g_2]](\Gn) e^{i k \Gvf}.
$$
Thus, 
\begin{align*}
\int_{\p \GO} f_1 \ol{\Scal_{\p \GO}[\Kcal^*_{\p \GO}[f_2]]}\,d\Gs =& \dfrac{{R_0}^3 (1 - \Gx^2)^2}{32 \pi} \int_0^{2\pi} g_1(\Gn) \ol{\Scal_k[\Acal_k[g_2]](\Gn)} \,d\Gn\\
=& \dfrac{{R_0}^3 (1 - \Gx^2)^2}{32 \pi} \la g_1, \Acal_k[g_2] \ra_k,
\end{align*}
and 
\begin{align*}
\nonumber \int_{\p \GO} \Kcal^*_{\p \GO}[f_1] \ol{\Scal_{\p \GO}[f_2]}\,d\Gs =&\dfrac{{R_0}^3 (1 - \Gx^2)^2}{32 \pi} \int_0^{2\pi} \Acal_k[g_1](\Gn) \ol{\Scal_k[g_2](\Gn)} \,d\Gn\\
=&\dfrac{{R_0}^3 (1 - \Gx^2)^2}{32 \pi} \la \Acal_k[g_1], g_2 \ra_k,
\end{align*}
from which \eqnref{symmetryk} follows. This completes the proof.
\qed

\section{Numerical range of $\Acal_k$ and the proof of Theorem \ref{main1}} \label{sec:main}

In this section, we prove the following theorem.

\begin{theorem}\label{main3}
For all $0 < \Gx < 1$, there exists a positive integer $k_0$ such that $\Acal_k$ has both positive and negative eigenvalues for all $k \in \Zbb$ with $|k| > k_0$.
\end{theorem}

Theorem \ref{main1} follows from Theorem \ref{main3}. In fact, by Lemma \ref{main2}, positive and negative eigenvalues of $\Acal_k$ yield positive and negative eigenvalues of $\Kcal_{\p \GO}^*$, respectively. Moreover, since eigenfunctions of $\Kcal_{\p \GO}^*$ take the form \eqnref{fgrel}, eigenfunctions corresponding to different $k$ are orthogonal to each other (see \eqnref{innerrela}). Since $\Kcal^*_{\p \GO}$ is compact, multiplicity of each eigenvalue is finite. Thus there must be infinitely many positive and negative eigenvalues.

To prove Theorem \ref{main3}, we show that the numerical range $\la \Acal_k[g], g \ra_k$ of $\Acal_k$ has both positive and negative values. Since $\Acal_k$ is self-adjoint, it means that there are both positive and negative eigenvalues.

Note that $a_k(\Gn, \Gn')$ can be written as
$$
a_k(\Gn, \Gn') = \psi(\Gn)^{-1} s_k(\Gn - \Gn') - \Gx \dfrac{\p}{\p \Gx} s_k(\Gn - \Gn').
$$
Thus, we have
\beq
\Acal_k[g](\Gn) = \Gy(\Gn)^{-1} \Scal_k[g](\Gn) - \Gx \dfrac{\p}{\p \Gx} \Scal_k[g](\Gn).
\eeq
Let $g_l(\Gn) := e^{i l \Gn}$. Then
\beq
\Scal_k[g_l](\Gn) = s_{k, l}(\Gx) e^{i l \Gn},
\eeq
where
\beq
s_{k, l}(\Gx) := \int_0^{2\pi} \int_0^{2\pi} \dfrac{e^{-i k \Gvf'} e^{-i l \Gn'}}{(\Gm(\Gvf') - \cos \Gn')^{1/2}}\,d\Gn' d\Gvf'.
\eeq
Thanks to \eqnref{equiv2}, we have
\beq\label{sklposi}
s_{k, l}(\Gx) >0 \quad\mbox{for all } k, l \in \Zbb, \ \ 0 < \Gx < 1.
\eeq

Since
\begin{align*}
\int_0^{2\pi} \frac{1}{\psi(\Gn)} \, d\Gn = \frac{2\pi}{\sqrt{1 - \Gx^2}},
\end{align*}
we have
\begin{align}\label{inner}
\la \Acal_k[g_l], g_l \ra_k =& s_{k, l}(\Gx) \int_0^{2\pi} \left[ \Gy(\Gn)^{-1} s_{k, l}(\Gx) - \Gx s_{k, l}'(\Gx) \right] \, d\Gn \nonumber \\
=& \frac{2\pi s_{k, l}(\Gx)}{\sqrt{1 - \Gx^2}} (s_{k, l}(\Gx) - \Gx \sqrt{1 - \Gx^2} s_{k, l}'(\Gx)).
\end{align}

We will investigate the sign of $\la \Acal_k[g_l], g_l \ra_k$. Thanks to \eqnref{sklposi}, it is enough to look into the quantity $I_{k,l}(\Gx)$ defined by
\beq
I_{k,l}(\Gx) := s_{k, l}(\Gx) - \Gx \sqrt{1 - \Gx^2} s_{k, l}'(\Gx).
\eeq
Observe that
\begin{align*}
I_{k,l}(\Gx)
= &\int_0^{2\pi} \int_0^{2\pi} \dfrac{1-\sqrt{1-\Gx^2} -(1 - \Gx^2 - \sqrt{1-\Gx^2} )\cos \Gvf  - \Gx^2 \cos{\Gn}}{\Gx^2 (\Gm(\Gvf) - \cos \Gn)^{3/2} } e^{-i k \Gvf} e^{-i l \Gn} \,d\Gn d\Gvf. \\
= &\int_{-\Gp}^{\Gp} \int_{-\Gp}^{\Gp} \dfrac{1-\sqrt{1-\Gx^2} -(1 - \Gx^2 - \sqrt{1-\Gx^2} )\cos \Gvf  - \Gx^2 \cos{\Gn}}{\Gx^2 (\Gm(\Gvf) - \cos \Gn)^{3/2} } e^{-i k \Gvf} e^{-i l \Gn} \,d\Gn d\Gvf.
\end{align*}
The second identity holds because the integrand is $2\Gp$-periodic with respect to both $\Gvf$ and $\Gn$. We also mention that
\beq\label{1000}
I_{k,l}(\Gx) = I_{-k,l}(\Gx) = I_{k,-l}(\Gx),
\eeq
so in what follows we only consider nonnegative $k$ and $l$.

To estimate $I_{k,l}(\Gx)$ we use the stationary phase method, which we recall now (see, e.g., \cite{H}).

\begin{theorem}[Stationary phase approximation]\label{SPM}
Let $D$ be a bounded domain in $\Rbb^d$, and let $h$ and $\Psi$ be $C^\infty$ functions on $\ol{D}$ such that all critical points of $\Psi$ are non-degenerate, i.e., the Hessian $H_\Psi(x_0)$ of $\Psi$ is non-singular at every $x_0 \in D$ such that $\nabla \Psi(x_0)=0$. Let $\Sigma$ be the set of critical points of $\Psi$. If there is no critical point of $\Psi$ on $\p D$, then the following asymptotic formula as $n \to \infty$ holds:
\begin{align}
    &\int_D  h(x) e^{i n \Psi(x)}\,dx \nonumber \\
    &= \sum_{x_0 \in \Sigma} e^{i n \Psi(x_0)} \left| \det H_\Psi(x_0) \right|^{- 1 / 2} e^{(i \pi / 4) \operatorname{sign}(H_\Psi(x_0))} \left( \frac{2 \pi}{n} \right)^{d / 2} h(x_0) + o(n^{- d / 2}) , \label{eq:4}
\end{align}
where $\operatorname{sign}(A)$ for a matrix $A$ is defined to be
$$
\operatorname{sign}(A) := \#\{ \mbox{positive eigenvalues of }A \} - \#\{ \mbox{negative eigenvalues of }A \}.
$$
\end{theorem}

\begin{proof}[Proof of Theorem \ref{main3}]
We rewrite the integral $I_{k, l}(\Gx)$ in terms of the polar coordinates. Let $D := (-\pi, \pi) \times (-\pi, \pi)$. We introduce the polar coordinates for $D$ by
$$
(\Gvf, \Gn) = (r \cos \Gt, r \sin \Gt), \quad 0 < r < R(\Gt), \quad -\pi \leq \Gt < \pi, \quad (\Gvf, \Gn) \in D,
$$
where
$$
R(\Gt) :=
\begin{cases}
\ds \frac{\pi}{|\cos \Gt|}, \quad &\mbox{if } \ds -\pi \leq \Gt < -\frac{3}{4} \pi, \  -\frac{1}{4} \pi \leq \Gt < \frac{1}{4} \pi, \ \frac{3}{4} \pi \leq \Gt < \pi,\\
\nm
\ds \frac{\pi}{|\sin \Gt|}, \quad & \mbox{otherwise}.
\end{cases}
$$
Then, we have
$$
I_{k, l}(\Gx) = \int_{-\pi}^{\pi} \int_0^{R(\Gt)} h(r, \Gt) e^{-i kr \cos \Gt} e^{-i lr \sin \Gt} \,dr d\Gt,
$$
where
$$
h(r, \Gt) :=
\begin{cases}
\dfrac{ |r| \{1-\sqrt{1-\Gx^2} -(1 - \Gx^2 - \sqrt{1-\Gx^2} )\cos (r \cos \Gt)  - \Gx^2 \cos(r \sin \Gt) \}}{\Gx^2 (\Gm(r \cos \Gt) - \cos ( r \sin \Gt))^{3/2} }, \quad r \neq 0,\\
\nm
\dfrac{\sqrt{2}\{\Gx(1 - \Gx^2 - \sqrt{1 - \Gx^2}) \cos^2 \Gt + \Gx^3 \sin^2 \Gt \}}{((1 - \Gx^2) \cos^2 \Gt + \Gx^2 \sin^2 \Gt)^{3/2}}, \quad r = 0.
\end{cases}
$$
One can easily see that $h$ is a $C^\infty$ function on $\Rbb^2$, an even function with respect to both $r$ and $\Gt$, and $\pi$-periodic in $\Gt$.

By changing variables of integration $r' = -r$ and $\Gt' = \Gt - \pi$, we have
$$
I_{k, l}(\Gx) = \int_{-\pi}^{\pi} \int_{-R(\Gt')}^0 h(r', \Gt') e^{-i kr' \cos \Gt'} e^{-i lr' \sin \Gt'} \,dr' d\Gt'.
$$
Here, we have made use of $\pi$-periodicity of $R$ and $h(r, \cdot)$. So, we also have
\beq \label{Ikl}
I_{k, l}(\Gx) = \frac{1}{2} \int_{-\pi}^{\pi} \int_{-R(\Gt)}^{R(\Gt)} h(r, \Gt) e^{-i kr \cos \Gt} e^{-i lr \sin \Gt} \,dr d\Gt.
\eeq

Now we are ready to investigate signs of the integral $I_{k, l}(\Gx)$. First assume that $l=0$. We apply Theorem \ref{SPM} with the phase function $\Psi(r, \Gt) = -r \cos \Gt$. The critical points of $\Psi$ in the region $\{ (r, \Gt) ~|~ -R(\Gt) < r < R(\Gt), \ -\pi \leq \Gt < \pi \}$ are $(0, \pm \pi/2)$. We further have
$$
H_\Psi (0, \pm \pi/2) =
\begin{bmatrix}
0 & \pm1\\
\pm1 & 0
\end{bmatrix},
$$
and hence
$$
|\det H_\Psi (0, \pm \pi/2)| = 1, \quad \operatorname{sign}(H_\Psi (0, \pm \pi/2)) = 0.
$$
Also, we have $h(0, \pm \pi/2) = \sqrt{2}$. It then follows from \eqnref{eq:4} that
\beq
I_{k,0}(\Gx) = \frac{2\sqrt{2}\Gp}{k} + o(1/k)
\eeq
as $k \rightarrow \infty$. Thus for each $0 < \Gx < 1$, there exists a positive integer $k_0$ depending on $\Gx$ such that
\beq\label{positive}
I_{k,0}(\Gx)>0
\eeq
for all $k > k_0$.

We next investigate the asymptotic behaviour of $I_{k,l}(\Gx)$ for fixed $k$ and large $l$. Let
$$
h_k(r, \Gt):= h(r,\Gt)e^{-i kr \cos \Gt},
$$
so that
$$
I_{k, l}(\Gx) = \frac{1}{2} \int_{-\pi}^{\pi} \int_{-R(\Gt)}^{R(\Gt)} h_k(r, \Gt) e^{-i lr \sin \Gt} \,dr d\Gt.
$$
We then make a change of variables $\Gt \to \Gt+\pi/2$ so that
$$
I_{k, l}(\Gx) = \frac{1}{2} \int_{-\pi}^{\pi} \int_{-R(\Gt)}^{R(\Gt)} h_k \left( r, \Gt+\frac{\pi}{2} \right) e^{-i lr \cos \Gt} \,dr d\Gt.
$$
Then the phase function is the same as before, namely, $\Psi(r, \Gt) = -r \cos \Gt$, and
$$
h_k \left( 0, \pm \frac{\pi}{2} + \frac{\pi}{2} \right) = \frac{-\sqrt{2} \Gx (1 - \sqrt{1 - \Gx^2})}{1 - \Gx^2},
$$
and hence
$$
I_{k,l}(\Gx) = -\frac{2\sqrt{2}\Gp\Gx(1 - \sqrt{1-\Gx^2})}{(1-\Gx^2)l} + o(1/l)
$$
as $l \rightarrow \infty$. Since $0 < \Gx < 1$, we have
$$
\frac{\Gx(1 - \sqrt{1-\Gx^2})}{(1-\Gx^2)} >0.
$$
Thus for each $0 < \Gx < 1$ and $k \in \Zbb$, there exists a positive integer $l_k$ such that
\beq\label{negative}
I_{k,l}(\Gx) < 0
\eeq
for all $l > l_k$.

The statement of the theorem follows from \eqnref{positive} and \eqnref{negative}.
\end{proof}


\end{document}